\documentclass[12pt]{article}
\usepackage{tikz}
\usepackage{amsmath}
\usepackage{amssymb}
\usepackage{amsthm}
\usepackage{fullpage}
\usepackage{verbatim}
\usepackage{stmaryrd}
\usepackage{authblk}
\usepackage{multirow}
\usepackage{mathrsfs}

\usepackage[colorlinks]{hyperref}
\usepackage{mathdots}
\usepackage{graphicx,subfigure}
\usepackage[english]{babel}
\usepackage{dsfont}
\usepackage[utf8]{inputenc}
\usepackage{tikz}
\usepackage{mathtools}
\usepackage{eulervm} 
\usepackage[T1]{fontenc}

\usepackage{tikz}
\usepackage{tikz-cd}

\usetikzlibrary{backgrounds,fit, matrix}
\usetikzlibrary{positioning}
\usetikzlibrary{calc,through,chains}
\usetikzlibrary{arrows,shapes,snakes,automata, petri}

\theoremstyle{plain}
\newtheorem{theorem}{Theorem}

\newtheorem{Lemma}[theorem]{Lemma}
\newtheorem{Proposition}[theorem]{Proposition}

\theoremstyle{definition}
\newtheorem{Definition}[theorem]{Definition}
\newtheorem{Definition-Remark}[theorem]{Definition/Remark}

\theoremstyle{remark}

\newtheorem{Remark}[theorem]{Remark}

\numberwithin{theorem}{subsection}

\newcommand{\C}{\mathbb{C}}
\newcommand{\N}{\mathbb{N}}

\newcommand{\Z}{\mathbb{Z}}
\newcommand{\Q}{\mathbb{Q}}

\newcommand{\mc}[1]{\mathcal{#1}} 

\newcommand*{\rarr}{\ensuremath{\rightarrow}}

\newcommand*{\mbb}[1]{\ensuremath{\mathbb{#1}}}

\newcommand*{\mcal}[1]{\ensuremath{\mathcal{#1}}}
\newcommand*{\mfrk}[1]{\ensuremath{\mathfrak{#1}}}

\newcommand*{\mtxt}[1]{\ensuremath{\text{#1}}}

\newcommand*{\defeq}{\mathrel{\vcenter{\baselineskip0.5ex \lineskiplimit0pt
                     \hbox{\scriptsize.}\hbox{\scriptsize.}}}=}

\DeclareMathOperator{\Der}{\mathsf{D}}
\DeclareMathOperator{\Tr}{\mathsf{tr}}
\DeclareMathOperator{\Tor}{\mathsf{Tor}}
\DeclareMathOperator{\RR}{\mathsf{R}}

\begin{document}

\title{Ansatz for $(-1)^{n-1}\nabla p_n$}

\author[1]{Soumya D Banerjee}
\author[2]{Mahir Bilen Can}
\author[3]{Adriano Garsia}
\affil[1]{{\small Tulane University; sbanerjee@tulane.edu}}
\affil[2]{{\small Tulane University, mahirbilencan@gmail.com}}    
\affil[3]{{\small UCSD; garsiaadriano@gmail.com}}

\normalsize

\date{\today}
\maketitle

\begin{abstract}

We construct a special family of equivariant coherent sheaves on the Hilbert scheme on $n$-points in the affine plane. 
The equivariant Euler characteristic of these sheaves are closely related to the symmetic functions $(-1)^{n-1} \nabla p_n$. 
We prove a higher cohomology vanishing result of these sheaves. 
It follows from the Bridgeland-King-Reid correspondence that there is an effective $S_n$ module underlying the aforementioned family of symmetric functions. 

\vspace{.2cm}

\end{abstract}

\maketitle

\section{Introduction}

Let $S_n$ denote the symmetric group of permutations of the set $\{1,\dots, n\}$, 
and let $\mathbf{x} \cup \mathbf{y}$ denote the set of $2n$ algebraically 
independent variables defined by
$$
\mathbf{x}:=\{x_1,\dots, x_n\}\ \text{ and }\  \mathbf{y}:=\{y_1,\dots, y_n\}.
$$
We consider the diagonal action of $S_n$ 
on $\C[\mathbf{x},\mathbf{y}]:=\C[x_1,\dots, x_n,y_1,\dots, y_n]$ given by 
\[
\sigma\cdot x_i = x_{\sigma(i)} \text{ and } \sigma\cdot y_i = y_{\sigma(i)}\qquad 
\text{for $i=1,\dots, n$ and $\sigma \in S_n$.}
\]
Let $I_+$ denote the ideal generated by the $S_n$-invariant homogeneous polynomials of positive degree
in $\C[\mathbf{x},\mathbf{y}]$. The {\em ring of diagonal coinvariants} is the quotient ring  
$$
R_n :=\C[\mathbf{x},\mathbf{y}] / I_+.
$$
Clearly, $R_n$ has the structure of an infinite dimensional graded $S_n$-module,
whose grading is induced by the bi-degrees of the elements of $\C[\mathbf{x},\mathbf{y}]$.  
It is known for sometime that the corresponding bi-graded Frobenius 
character of $R_n$ is given by the Bergeron-Garsia's ``nabla'' operator, $\nabla$  
applied to the $n$-th elementary symmetric function, $e_n$. 
Roughly speaking, $\nabla$ is a $\mbb{Q}$-linear operator 
defined on the ring of symmetric functions, denoted by $\Lambda$, such that the 
`modified Macdonald symmetric functions' are eigenfunctions with prescribed eigenvalues
(see Section~\ref{S:Prelim} for details).
We say that 
a symmetric function $f \in \Lambda$ has a representation if 
$f$ is the Frobenius character of an $S_n$-module $F$. 
In particular, $\nabla e_n$ has the representation $F=R_n$.
It is a natural question to ask if $\nabla f$ has a representation
for any other well known symmetric function.
In general, for arbitrary $f\in \Lambda$, this question is difficult to answer. 
A prerequisite to have a positive answer is that $f$ is {\em Schur-positive},
that is to say, the expansion of $f$ in the Schur 
basis for $\Lambda$ has nonnegative integral coefficients.
We known from the work of Loehr and Warrington in~\cite{LoehrWarrington} that 
the symmetric functions $(-1)^{n-1}\nabla p_{n}$ ($n\in \N$) are Schur-positive. 
The purpose of our article is to construct, in an algebro-geometric way, 
a representation for each 
element of the set $\{(-1)^{n-1}\nabla p_{n}:\ n\in \N\}$.
We will use the techniques that are developed by Mark Haiman to prove 
the $n!$ conjecture (see~\cite{HaimanSurvey} for an overview).
\\

The fundamental geometric object in this story is the 
Hilbert scheme of $n$ points in the complex affine plane, 
which is denoted by $H_{n}$. 
It is a (fine) moduli space, parametrizing the ideals of colength $n$ in the coordinate ring of $\C^2$
The closed points of the scheme $H_{n}$ are given by
\begin{align}\label{A:variety}
H_n(\C) := \lbrace  I:\; I \subset \C[x,y] \text{ is an ideal such that } \dim_{\C} \C[x,y]/I = n  \rbrace.
\end{align}
Although it is not completely obvious from (\ref{A:variety}), 
$H_{n}$ turns out to be a smooth irreducible $2n$-dimensional algebraic variety over $\C$.
To set-up our notation and to further motivate our discussion, 
we will briefly review some important categories attached to $H_n$.

The Hilbert-Chow map,
\begin{align} \label{E:HilbertChow}
\sigma: H_n \rightarrow \C^{2n}/S_{n}
\end{align}
describes $H_{n}$ as a relative projective scheme over $\C^{2n}/S_{n}$,
and moreover, (\ref{E:HilbertChow}) turns out to be a crepant resolution of the isolated singularity. 
Consequently, the Bridgeland-King-Reid theorem provides a categorical equivalence, 
\begin{align}\label{E:catequivalence}
\Phi: \mathsf{D}^{b}(H_{n}) \longrightarrow \mathsf{D}^{b}_{S_{n}}(\C^{2n}),
\end{align}
between the bounded derived category of $S_{n}$-equivariant sheaves on $\C^{2n}$ 
and the bounded derived category of coherent sheaves on $H_{n}$.
Let $T$ denote the algebraic torus $\C^*\times \C^*$ with coordinates $(q,t)$.
The coordinatewise multiplication action of $T$ on $\C^2$ gives an action $T\times H_n\rightarrow H_n$.
Along with (\ref{E:catequivalence}), this torus action induces an isomorphism, also denoted by $\Phi$, 
between the equivariant Grothendieck groups,
\begin{align} \label{E:ktheoryisom}
\Phi: K^{0}_{T}(H_{n}) \cong K^{0}_{S_{n \times T}}(\C^{2n}).
\end{align}
By (\ref{E:ktheoryisom}), the questions about $S_{n}$-modules
translate into geometric questions about the equivariant sheaves on  $H_{n}$. 
More precisely, the $T$-equivariant Euler-characteristic of a sheaf
$\mcal{F}$, computed with the canonical polarization coming from (\ref{E:HilbertChow}), 
equals the Frobenius character of the $S_{n}$-representation $\Phi(\mcal{F})$. 
Indeed, $K^{0}_{S_{n \times T}}(\C^{2n})$ is freely generated, 
as a $\Z[q,t,q^{-1},t^{-1}]$-module, by the free 
$\C[\mathbf{x},\mathbf{y}]$-modules 
$V^\lambda \otimes \C[\mathbf{x},\mathbf{y}]$,
where $\lambda$ runs over all partitions of $n$, and
$V^\lambda$ is the irreducible $S_n$-module
corresponding to $\lambda$. 
Now, by applying the Euler-Frobenius characteristic  
map, it is not difficult to see that 
$K^{0}_{T}(H_{n})$ is isomorphic to the algebra 
of symmetric functions $f$ with the property that 
$f[(1-q)(1-t)Z]$ has coefficients in $\Z[q,t,q^{-1},t^{-1}]$.
Here, $(1-q)(1-t)Z$ is the symmetric function 
$(1-q)(1-t) (z_1+z_2+\cdots )$ and the brackets in
$f[(1-q)(1-t)Z]$ indicates that we are applying the 
``plethystic substitution.'' 

The  {\em isospectral Hilbert
scheme}, denoted by $X_n$, is the reduced fiber product 
$X_n := H_n \times_{\C^{2n}/S_n} \C^{2n}$, 
defined as in Figure~\ref{A:F1}.
\begin{figure}[h]
\begin{center}
\begin{tikzpicture}[scale=.75]
\node at (-2.25,1.5) (a) {$X_n$};
\node at (2.25,1.5) (b) {$\C^{2n}$};
\node at (-2.25,-1) (c) {$H_n$};
\node at (2.25,-1) (d) {$\C^{2n}/S_n$}; 
\node at (-2.7,0) {}; 
\node at (2.6,0) {}; 
\node at (2.75,.35) {$\pi$}; 
\node at (0,-.5) {$\sigma$}; 
\node at (-2.75,.35) {$\rho$}; 
\draw[->,thick] (a) to (b);
\draw[->,thick] (a) to (c);
\draw[->,thick] (b) to (d);
\draw[->,thick] (c) to (d);
\end{tikzpicture}
\end{center}
\label{A:F1}
\caption{The isospectral Hilbert scheme}
\end{figure}

In \cite{HaimanJams}, Haiman showed that the coherent sheaf $\mcal{P} :=
\rho_{\ast}(\mcal{O}_{X_{n}})$ is a locally free sheaf of rank $n!$ on $H_{n}$. 
The resulting vector bundle, called the {\em Procesi bundle}, turns out to be of fundamental importance. 
In particular, the class of $\mcal{P}|_{Z_{n}}$ in $K^{0}_{T}(H_{n})$, where $Z_{n}$ is the reduced
fiber $\sigma^{-1}( \pi(0))$ in $H_{n}$, corresponds to the representation underlying $\nabla e_{n}$
 under the isomorphism  $\Phi$. 
We are now ready explain the main result of our note.

We describe two $T$-equivariant coherent 
sheaves $\mcal{F}$ and $\mcal{F}^{\prime}$ on $H_n$. 
A priori, the first sheaf is an object of $\Der^{b}_{T}(H_{n}, \Z)$, 
and the second sheaf lives in $\Der^{b}_{T}(H_{n}, \Z[1/n])$. 
Roughly speaking, these sheaves are the subsheaves of some
standard vector bundles on $H_{n}$ twisted with suitable \'{e}tale local-systems coming from the torus $T$. 
By using a version of the Atiyah-Bott-Lefschetz theorem, 
we calculate the $T$-equivariant Euler-characteristic of these sheaves. 
We show for the first sheaf that 
the resulting expression is the Frobenius series of $(-1)^{n-1}n \nabla p_{n}$, and for the
second sheaf, it is $(-1)^{n-1}\nabla p_{n}$. 
We recorded these facts as Theorem~\ref{t:main} in Section~\ref{S:Prelim}.

The structure of our paper is as follows. In Section~\ref{S:Prelim}, 
we set up notation to explain the strategy of the proof of our Theorem~\ref{t:main}.
Also in the same section, we briefly review some relevant results of Haiman, Loehr and Warrington.
The final Section~\ref{S:Proof} is devoted to the proof of the theorem.

\subsubsection*{Acknowledgements}
The authors thank Eugene Gorsky for pointing out a crucial error in the first version of this paper. 
The authors thank Al Vitter for very helpful discussions.

\section{Preliminaries}\label{S:Prelim}

A partition $\mu$ of $n$ is a nonincreasing sequence 
of nonnegative integers $\mu=(\mu_1,\mu_2,\dots)$,
with $\sum_{i\geq 1} \mu_i = n$. 
We will use the French notation to represent the diagram of 
partitions as in Figure~\ref{F:F}. Therefore, there are 
$\mu_1$ squares in the bottom row, there are $\mu_2$ 
squares in the second row, and so on.  
The squares in the diagram of $\mu$ 
will be called the cells of the partition.
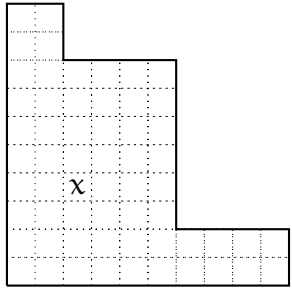
\begin{figure}[h]
\begin{center}
\begin{tikzpicture}[scale=.75]

\node at (1.25,1.75) {$x$}; 

\draw[-,thick] (0,0) -- (0,5) -- (1,5) -- (1,4) -- (3,4) -- (3,1) -- (5,1) -- (5,0) -- (0,0);
\draw[dotted] (0.5,0) -- (.5,5);
\draw[dotted] (1,0) rectangle (1,4);
\draw[dotted] (1.5,0) rectangle (1.5,4);
\draw[dotted] (2,0) rectangle (2,4);
\draw[dotted] (2.5,0) rectangle (2.5,4);
\draw[dotted] (3,0) rectangle (3,1);
\draw[dotted] (3.5,0) rectangle (3.5,1);
\draw[dotted] (4,0) rectangle (4,1);
\draw[dotted] (4.5,0) rectangle (4.5,1);
\draw[dotted] (5,0) rectangle (5,1);

\draw[dotted] (0,0.5) rectangle (5,.5);
\draw[dotted] (0,1) rectangle (4,1);
\draw[dotted] (0,1.5) rectangle (3,1.5);

\draw[dotted] (0,2) rectangle (3,2);
\draw[dotted] (0,2.5) rectangle (3,2.5);
\draw[dotted] (0,3) rectangle (3,3);
\draw[dotted] (0,3.5) rectangle (3,3.5);
\draw[dotted] (0,4) rectangle (1,4);
\draw[dotted] (0,4.5) rectangle (1,4.5);

\end{tikzpicture}
\end{center}
\caption{The diagram of a partition in French notation.}
\label{F:F}
\end{figure}

Let $\mu$ be a partition of $n$ with parts $\mu = (\mu_1,\mu_2,\dots )$.
In this case, we write $\mu \vdash n$. 
The partition that is conjugate to $\mu$ is denoted by
$\mu^\prime$.
We will use $n(\mu)$ to denote the following sum:
$$
n(\mu) = \sum_i (i-1) \mu_i.
$$
The dominance order on partitions, 
denoted by $\leq$, is defined by 
$$
(\mu_1,\mu_2,\dots) \leq (\lambda_1,\lambda_2,\dots) \iff 
\mu_1+ \dots +\mu_i \leq \lambda_1+ \dots +\lambda_i \ 
\text{ for all } i \geq 1.
$$

Let $x$ be a cell in $\mu$. 
The {\em arm, coarm, leg}, and the {\em coleg of $x$} are defined by  
\begin{enumerate}
          \item $a(x) = $ arm: the number of cells that are directly above $x$;  
          \item $a^{\prime}(x) = $ coarm: the number cells that are directly below $x$;
          \item $l(x) = $ leg: the number of cells to the right of $x$; 
          \item $l^{\prime}(x) = $ coleg:  the number of cells to the left of $x$.
\end{enumerate}          
For example, the arm, coarm, leg, and the coleg of $x$ in Figure~\ref{F:F} 
are 4,3,3, and 2, respectively.
In this notation, we identify the cells $x$ 
in the partition $\mu$ with their ``cartesian coordinates'' 
$(a^\prime (x) , l^\prime (x))$. 
We denote by $\square_{n}$ 
the partition with $n$ equal parts 
all of which equals $n$,
$$
\square_{n} := \{ (r,s) \in \Z \times \Z \; | 0 \leq r,s \leq n-1 \}.
$$
In addition, we define  
\begin{enumerate}
\item  $[t]_{n} := 1 + t + \ldots+ t^{n-1}$ and $[q]_{n} := 1 + q + \ldots+ q^{n-1}$;
\item $B_{\mu} := \sum_{x \in \mu} q^{a^{\prime}(x)}t^{l^{\prime}(x)}$;
\item $ T_{\mu} := q^{n(\mu^{\prime})}t^{n(\mu)}$;
\item $\Pi_{\mu} := \prod_{x \in \mu \setminus \{ 0,0 \}} (1 - q^{a^{\prime}(x)}t^{l^{\prime}(c)})$.
\end{enumerate}

Let $\Lambda$ denote the algebra of symmetric functions
in the variables $z=z_1,z_2,\dots $ with coefficients in $\Q(q,t)$. 
Then $\Lambda$ is a graded algebra, graded by the degree of the symmetric functions, 
$\Lambda = \bigoplus_{n\geq 0} \Lambda_n$. 
It is well known that the elementary symmetric functions, power-sum symmetric functions,
and the Schur functions, are $\Q$-bases for $\Lambda$ (see~\cite[Chapter I]{Macdonald}).

The {\em modified Macdonald symmetric functions}, which form a $\Q(q,t)$-basis for $\Lambda$, 
are uniquely determined by the following three properties:
\begin{enumerate}
\item $\widetilde{H}_\mu (z;q,t) \in \Q(q,t) 
\{ s_\lambda \lbrack Z/(1-q) \rbrack :\ \mu \leq \lambda \}$,
\item $\widetilde{H}_\mu (z;q,t) \in \Q(q,t) 
\{ s_\lambda \lbrack Z/(1-q) \rbrack :\ \mu' \leq \lambda \}$,
\item $\widetilde{H}_\mu \lbrack 1;q,t \rbrack =1$,
\end{enumerate}
where the brackets indicate the plethystic substitution,
$Z/(1-q)$ is the symmetric function 
$\frac{1}{1-q} (z_1+z_2+\cdots )$, and 
$s_\lambda(z)$ denotes a Schur function. 
The Bergeron-Garsia operator on $\Lambda$ is explicitly given by setting  
\[
\nabla ( \widetilde{H}_\mu(z;q,t) ) 
= t^{n(\mu)} q^{ n(\mu^\prime)} \widetilde{H}_\mu(z;q,t).
\]It turns out that $\nabla e_n$ is the Frobenius series 
of the ring of diagonal coinvariants $R_n$ and 
its expansion in the (modified) Macdonald basis is given by 
\begin{align}\label{E:en}
\nabla e_n =    \sum_{\mu \vdash n} \left( \frac{ (1-q)(1-t) \Pi_{\mu} B_{\mu}}
{ \prod_{x \in \mu} (1 - t^{1+l(x)}q^{-a(x)})(1 - t^{-l(x)}q^{1+ a(x)})} \right) \widetilde{H}_{\mu}(z;q,t),
\end{align}
see~\cite[Proposition 3.5.26]{HaimanSurvey}.

In \cite{LoehrWarrington}, Loehr and Warrington obtained an expansion of the symmetric functions $(1)^{n-1} \nabla
p_{n}$ in the (modified) Macdonald basis,
\begin{align}\label{E:pn}
(-1)^{n-1}\nabla p_{n}= \sum_{\mu \vdash n} \left (\frac{(1-t^{n})(1- q^{n}) \Pi_{\mu} T_{\mu}}
{\prod_{x \in \mu}(q^{a(x)} - t^{1+ l(x)})(t^{l(x)} - q^{1+ a(x)}) } \right) \widetilde{H}_{\mu}(z;q,t).
\end{align}
By using the identities
\begin{align*}
  n(\mu) = \sum_{x \in \mu} l(x) \ \text{ and }\ n(\mu^{\prime}) = \sum_{x \in \mu^{\prime}} l(x) = \sum_{x \in \mu} a(x),
\end{align*}
let us rewrite \eqref{E:pn} in a form that is similar to \eqref{E:en};
\begin{align}\label{E:refinedpn}
(-1)^{n-1}\nabla p_{n}= \sum_{\mu \vdash n} \left (\frac{(1-t)(1- q) \Pi_{\mu} [t]_{n}[q]_{n}}
{\prod_{x \in \mu}(1 - t^{1+l(x)}q^{-a(x)})(1 - t^{-l(x)} q^{1+a(x)}) } \right) \widetilde{H}_{\mu}(z;q,t).
\end{align}

The purpose of this note is to prove the following theorem.

\begin{theorem}\label{t:main}
  There exists a coherent sheaf $\mathcal{F}$ (resp. $\mcal{F}^{\prime}$) on the Hilbert scheme $H_n$ 
  such that the Frobenius series of the $S_n$-module under the derived equivalence  $\Phi(\mcal{F} \otimes \mcal{P})$ (resp. $\Phi(\mcal{F}^\prime \otimes \mcal{P})$) is given by $(-1)^{n-1}n \nabla p_{n}$ (resp. $(-1)^{n-1}\nabla p_{n}$).
\end{theorem}

Let us summarize the major steps of the proof.
\begin{enumerate}

\item We present a construction of $\mcal{F}$ (resp. $\mcal{F}^\prime$).

\item By using the Atiyah-Bott-Lefschetz formula, we calculate the 
equivariant Euler characteristics of the sheaf 
$\mcal{F} \otimes \mcal{P}$ (and of $\mcal{F}^\prime \otimes \mcal{P}$), and 
we show that it has the expected form as in (\ref{E:refinedpn}).

\item We show that the higher cohomology of $\mcal{F}$ (resp. of $\mcal{F}^\prime$) vanishes. 
This proves that under the Bridgeland-King-Reid correspondence the resulting
  $S_n$-module, which is a priori a virtual representation, 
is an honest $S_n$-representation.
\end{enumerate}

\subsection{Some useful facts about Hilbert scheme of points on the plane}

In this subsection, we will recall some basic facts about $H_{n}$ that we will use in the sequel. 
For the detailed proofs of these facts, we refer to~\cite[Section 2]{HaimanDiscrete}.

The two-dimensional torus $T$ with coordinates $(q,t)$ acts on $\C^{2}$.
On the coordinate ring $\C[x,y]$ of $\C^2$, this action is given by $(q,t)\cdot x = qx$ and $(q,t)\cdot y = ty$. 
The induced action on closed points of $H_{n}$ is given by the pull-back of ideals.
Explicitly, if $I = \langle p(x,y) \rangle$, then 
$$
(q,t)\cdot I = \langle p(q^{-1}x, t^{-1}y) \rangle.
$$

We know from the Gotzmann-Grothendieck construction that $H_n$ has an imbedding into a 
high dimensional Grassmann variety. 
By pulling back the standard open covering of the Grassmann variety, 
we get an open cover $\{U_{\mu}\}$ of $H_{n}$ indexed by the partitions of $n$. 
Explicitly, the closed points of $U_{\mu}$ are given by 
\begin{align}\label{E:standardopen}
U_\mu := \{ I \in H_n:\ \C[x,y]/I \ \text{admits a vector-space basis of monomials } x^{h}y^{k} \mtxt{ for } (h,k) \in
  \mu \}.
\end{align}
The coordinate ring $\mcal{O}_{U_{\mu}}$  is generated by functions  $c^{rs}_{hk}$. 
On an ideal $I \in U_{\mu}$, these functions take the values
\[ 
x^{r}y^{s} = \sum_{(h,k) \in \mu} c^{rs}_{hk}(I) x^{h}y^{k} \mod (I)
\] 
for all $r,s \geq 0$. 
As a result, the $T$-action on $c^{rs}_{hk}$ is given by $(q,t) c^{rs}_{hk} = q^{r-h}t^{s-k}c^{rs}_{hk}$. 
It follows from the definitions that the above open covering is stable under the $T$-action. 
Furthermore, the $T$-fixed points of $H_{n}$ are precisely the monomial ideals 
\[
I_\mu := \langle x^h y^k :\ (h,k) \notin \mu \rangle.
\] 
The maximal ideal $\mfrk{m}_{\mu}$ at $I_{\mu}$ is described by
\[ 
\mfrk{m}_{\mu} = \{ c^{rs}_{hk} | \; (r,s) \notin \mu \}.
\]

Let $\pi: F_{n}\rightarrow H_n$ denote the universal family over $H_{n}$. 
We set 
$ \mcal{B}:= \pi_{\ast} \mcal{O}_{F_{n}}.$ 
The sheaf $\mcal{B}$ is locally free, and it has rank $n$. 
The trace map
\[ 
\Tr: \mcal{B} \rarr \mcal{O}_{H_{n}}
\] 
splits the canonical inclusion $ \mcal{O}_{H_{n}} \rightarrow  \mcal{B}$, and hence, 
we obtain an equivariant splitting of vector bundles,
\[ 
\mcal{B} = \mcal{B}^{\prime} \oplus \mcal{O}_{H_{n}}.
\]


\begin{Remark}
Following Haiman, we will call $\mcal{B}$ the {\em tautological bundle} on $H_{n}$. 
Note that the Hilbert-Chow map $\sigma: H_{n} \rarr \C^{2n}/S_{n}$ is a blow-up of $\C^{2n}/S_{n}$. 
This provides $H_{n}$ with a canonical ample line bundle $\mcal{O}(1)$. 
It turns out that $\wedge^{n} \mcal{B} = \mcal{O}(1)$ (see~\cite{HaimanDiscrete}).
\end{Remark}

Let $[0]$ denote the image of origin $0\in \C^{2n}$ under the canonical projection 
$\C^{2n} \rarr \C^{2n}/S_{n}$. 
The \emph{zero fiber}, denoted by $Z_{n}$ is the closed subscheme $ \sigma^{-1}([0])$. 
By Brian\c{c}on's work~\cite{Briancon}, we know that $Z_{n}$ is an irreducible scheme of dimension
$n-1$. In~\cite{HaimanDiscrete}, Haiman showed that the structure sheaf $\mcal{O}_{Z_{n}}$ is $T$-equivariantly
isomorphic to a perfect complex in derived category of coherent sheaves in $H_{n}$. More precisely, he proved 
the following result. 
\begin{theorem}
    In the derived category of $T$-equivariant coherent sheaves on $H_{n}$, the sheaf $\mcal{O}_{Z_{n}}$ is isomorphic to
 \begin{align}\label{E:KoszulResolution}
   0 \rightarrow \mc{B} \otimes \wedge^{n+1} (\mc{B}^\prime \oplus \mc{O}_t \oplus \mc{O}_q)\rightarrow \dots \rightarrow 
\mc{B} \otimes (\mc{B}^\prime \oplus \mc{O}_t \oplus \mcal{O}_q)
\rightarrow \mc{B} \rarr 0,
 \end{align}
where all tensor products are over $\mcal{O}_{H_{n}}$ and $\mcal{O}_{t}$ (resp. $\mcal{O}_{q}$) denote trivial  line bundles
 on $H_{n}$ with $T$-characters $t$ (resp. $q$).
\end{theorem}

\section{Proof of the Theorem} \label{S:Proof}
\subsection{Step 1.}

Let $p_1$ and $p_2$ denote the canonical projections as in Figure~\ref{D:label}. 
\begin{figure}[h]
\begin{center}
\begin{tikzpicture}[scale=.75]
\node at (-2.25,1.5) (a) {$H_n \times T$};
\node at (1.75,1.5) (b) {$T$};
\node at (-2,-1.65) (c) {$H_n$};

\node at (0,2) {$p_2$}; 
\node at (-2.85,0) {$p_1$}; 
\draw[->,thick] (a) to (b);
\draw[->,thick] (-2.25,1.1) to (-2.25,-1.1);
\end{tikzpicture}
\end{center}
\label{D:label}
\caption{}
\end{figure}

\begin{Lemma} 
If $\mcal{F}$ is a quasi-coherent sheaf on $H_{n} \times T$, 
then we have $R^{i}p_{\ast}(\mcal{F}) = 0 $ for all $i > 0$.
\end{Lemma}

\begin{proof}
    This follows directly from Corollary I.3.2 Chapter III of  \cite{EGA}.
\end{proof}

We will interchangeably use the same letter to denote a vector bundle and the associated sheaf of sections. 
Let $\chi$ be a character of the torus $T$, and let $\mcal{O}_{\chi}$ denote the associated line bundle on $T$. 
Then we define $L_\chi$ as follows;
\[ 
L_{\chi} := p_{1,\ast}\circ p_{2}^{\ast}(\mcal{O}_{\chi}). 
\]

\begin{Proposition}
The quasi-coherent sheaf $L_{\chi}$ is a $T$-equivariant line bundle on $H_{n}$.
\end{Proposition}

\begin{proof}

First, we will show that $L_{\chi}$ is a line bundle. 
But this is clear since for any open subscheme $U$ of $H_{n}$, 
we have sections
\[
L_{\chi}(U) := \mcal{O}_{H_{n}}(U) \otimes_{\C} \C \cdot \chi. 
\]
This shows that, indeed, $L_{\chi}$ is a line bundle. 

To show that it is $T$-equivariant, we will construct an isomorphism
$\varphi: p_{1}^{\ast}L_{\chi} \rarr a^{\ast} L_{\chi}$, where $a:H_{n} \times T \rarr H_{n}$ is the action
map. We use the $T$-stable affine-covering of $\{ U_{\mu}\}$ of $H_{n}$ constructed as above. 
Restricted to the open set $U_{\mu}$, we have the canonical identifications 
$\varphi_{\mu}: p_{1}^{\ast}L_{\chi}(U_{\mu}) \cong a^{\ast} L_{\chi}(U_{\mu})$. 
These sections agree on the intersections $U_{\mu} \cap U_{\lambda}$, and, hence, 
they glue to a global isomorphism of sheaves. This proves our assertion. 
\end{proof}

For a cell $(r,s)$ in the square grid $\square_{n}$, let us denote by $\chi_{r,s}$ the character 
$\chi_{r,s} : T\rightarrow \C^{\ast}$ defined by 
\[
\chi_{r,s} (q,t) := q^{r}t^{s}.
\]
\begin{Definition} \label{D:VandN}
Let 
\[
\mathcal{V} := \bigoplus_{(r,s) \in \square_{n}} L_{\chi_{r,s}}
\]
denote a $T$-equivariant coherent sheaf on $H_{n}$.
\end{Definition}

\subsubsection{A variation: the construction of $\mcal{V}^{\prime}$.}

Let $\iota$ denote the following isogeny:
\begin{align}
\iota: T &\longrightarrow T \notag \\
(q,t) &\longmapsto (q^{n},t^{n}). \label{A:isogeny}
\end{align}
It fits to a commutative diagram as in Figure~\ref{F:2}.
Then, for a line bundle $\mcal{O}_{\chi}$ on $T$, we define 
$$
L^{\prime}_{\chi} := p_{1, \ast} \circ \tau^{\ast} \mcal{O}_{\chi}.
$$

\begin{figure}[htbp]
\begin{center}
\begin{tikzpicture}[scale=.75]
\node at (-2.25,1.5) (a) {$H_n \times T$};
\node at (1.75,1.5) (b) {$T$};
\node at (1.75,-1.65) (c) {$T$};
\node at (-2.25,-1.65) (d) {$H_n$};
\node at (-2.75,0) {$p_1$}; 
\node at (0,2) {$p_2$}; 
\node at (-.5,-.5) {$\tau$}; 
\node at (2.2,0) {$\iota$}; 
\draw[->,thick] (a) to (c);
\draw[->,thick] (a) to (b);
\draw[->,thick] (b) to (c);
\draw[->,thick] (a) to (d);
\end{tikzpicture}
\end{center}
\caption{}
\label{F:2}
\end{figure}

By using the same arguments as before,  
we can show that $L_{\chi}^{\prime}$ is a $T$-equivariant line bundle on $H_{n}$. 
We define, analogous to $\mcal{V}$, a $T$-equivariant coherent sheaf 
\begin{align*}
\mathcal{V}^{\prime} &:= \oplus_{(r,s) \in \square_{n}} L^{\prime}_{\chi_{r,s}}.
\end{align*}

\subsubsection{The descriptions of $\mcal{F}$ and $\mcal{F}^\prime$.}

Let $\mcal{K}^\bullet$ denote the complex
\begin{align*}
  0 \rightarrow\mcal{V}\otimes \mc{B} \otimes \wedge^{n+1} (\mc{B}^\prime \oplus \mc{O}_t \oplus \mc{O}_q)
  \stackrel{\partial_{n+1}}{\rightarrow} \dots \stackrel{\partial_{2}}{ \rightarrow} 
\mcal{V} \otimes\mc{B} \otimes (\mc{B}^\prime \oplus \mc{O}_t \oplus \mcal{O}_q)
\stackrel{\partial_{1}}{\rightarrow} \mcal{V} \otimes \mc{B} \rarr 0. 
\end{align*}
For any positive integer $j$, let $\partial^\prime_i$ denote the composition $i \circ \partial_i \circ \Tr$ in the diagram
\eqref{E:differential} below. 
\begin{equation} \label{E:differential}
\begin{tikzcd}
    \mcal{V} \otimes \mcal{O}_{H_n} \otimes \wedge^{j+1} (\mc{B}^\prime \oplus \mc{O}_t \oplus \mc{O}_q)  \arrow[d,"i"]
    \arrow[r, dashed, "\partial^\prime_i"] & \mcal{V} \otimes \mcal{O}_{H_n} \otimes \wedge^{j} (\mc{B}^\prime \oplus \mc{O}_t \oplus \mc{O}_q) \\
    \mcal{V} \otimes \mcal{B} \otimes \wedge^{j+1} (\mc{B}^\prime \oplus \mc{O}_t \oplus \mc{O}_q) \arrow[r,"\partial_i"] & \mcal{V} \otimes \mcal{B} \otimes \wedge^{j} (\mc{B}^\prime \oplus \mc{O}_t \oplus \mc{O}_q) \arrow[u,"\Tr"]
\end{tikzcd}
\end{equation}
In \eqref{E:differential}, $i$ is induced by the canonical push-forward 
$\mcal{O}_{H_{n}} \hookrightarrow \mcal{B}$, and $\Tr$ is induced by the
trace map $\Tr: \mcal{B} \rarr \mcal{O}_{H_n}$. 
The following proposition is clear from straightforward computation.

\begin{Proposition}
We have $\partial^\prime_i \circ \partial^\prime_{i-1} = 0$. 
By putting the maps $\partial^\prime_i$ together, 
we get a bounded  complex $\mcal{S}^\bullet$ given by
  \[
  \mcal{S}^\bullet \defeq  0 \rightarrow\mcal{V}\otimes \mc{O}_{H_n} \otimes \wedge^{n+1} (\mc{B}^\prime \oplus \mc{O}_t \oplus \mc{O}_q)
\rarr \dots \rarr \mcal{V} \otimes\mc{O}_{H_n} \otimes (\mc{B}^\prime \oplus \mc{O}_t \oplus \mcal{O}_q) \rarr \mcal{V} \otimes \mc{O}_{H_n} \rarr 0. 
\]
Moreover, we have the chain maps $i: \mcal{S}^\bullet \rarr \mcal{K}^\bullet$ and $\Tr: \mcal{K}^\bullet \rarr
\mcal{S}^\bullet$.  Finally, $\Tr \circ i $ is the identity map on $\mcal{S}^\bullet$.
\end{Proposition}

In fact, the following is true.
\begin{Proposition}
  In the bounded derived category of equivariant sheaves, the complex $\mcal{S}^\bullet$ is a direct summand of
  $\mcal{K}^\bullet$.
\end{Proposition}

\begin{proof} 
This follows from Lemma 1.2.8 of \cite{ANeeman} and the previous proposition.
\end{proof}

The complex $\mcal{K}^\bullet$ is exact except at degree zero. 
Consequently, the complex $\mcal{S}^\bullet$ is also exact except possibly at degree zero. 
In fact, from the description of the resolution \eqref{E:KoszulResolution}, 
it follows that $\mcal{S}^\bullet$ is not exact at degree zero 
(see the discussion after Proposition 2.2 in~\cite{HaimanInven}.)

\begin{Definition}
Let $\mcal{F}$ denote the image of $\mcal{S}^\bullet$ inside $\mcal{O}_{Z_n} \otimes  \mcal{V}$ 
under the quasi-isomorphism $\mcal{K}^\bullet \rarr \mcal{O}_{Z_n} \otimes \mcal{V}$.
In the above construction, if $\mcal{V}^\prime$ is used instead of $\mcal{V}$, 
then let $\mcal{F}^\prime$ denote the corresponding image.
\end{Definition}


\subsection{Step 2.}

In this subsection, we will calculate the equivariant Euler characteristic of the sheaf $\mcal{F} \otimes \mcal{P}$ 
(resp. $\mcal{F}^{\prime} \otimes \mcal{P}$) by using an appropriate version of the Atiyah-Bott-Lefschetz
formula.

We know from~\cite[(3.1)]{HaimanDiscrete} that
\begin{align}\label{A:H's formula}
\sum_{i=0}^{n} (-1)^{i} \Tr_{H^{i}(H_{n}, \mcal{E})} (\tau) = \sum_{\mu \vdash n} \frac{\sum_{i=0}^{n} (-1)^{i}
  \Tr_{\Tor_{i}(k(I_{\mu}), \mcal{E})}(\tau)}{\det_{C(I_{\mu})}(1- \tau)},
\end{align}
where $\mcal{E}$ is any bounded complex of coherent sheaves, 
$\tau \in T$, $k(I_\mu)$ is the residue field and $C(I_{\mu})$ is the tangent space the $T$-fixed point $I_{\mu}$.
We will evaluate the right hand side of (\ref{A:H's formula}) for the sheaf $\mcal{E} = \mcal{F} \otimes \mcal{P}$. 
Note that the terms in the denominator, namely $\det_{C(I_{\mu})}(1 - \tau)$, are well known from the
work of Ellingsrud and Str\o mme (see~\cite{Ellingsrud}). 
So, we will calculate the numerator. 
The tor-modules $\Tor_i(k(I_\mu),\mcal{F} \otimes \mcal{P})$ are obtained as the 
homology of the complex $\mcal{S}^\bullet\otimes \mcal{P} \otimes k(I_\mu)$. 
By using the multiplicative property of the trace map with respect to tensor products, 
we obtain  
\begin{align}\label{E:CalculatingTraces}
\sum_{i} (-1)^{i}\Tr_{\Tor_{i}(k(I_{\mu}), \mcal{F} \otimes \mcal{P})}(\tau) = 
\Tr_{k(I_{\mu})}(\tau) \cdot \Tr_{\mcal{P}_{I_\mu}}(\tau) \cdot \Tr_{\mcal{V}_{I_\mu}}(\tau) \cdot \left( \sum_{i}(-1)^{i}
  \Tr_{\wedge^{i}(\mcal{B}^{\prime} \oplus \mcal{O}_{t} \oplus \mcal{O}_{q})|_{I_{\mu}}} \right )
\end{align}

Next, we look at the explicit formulation of the right hand side of (\ref{E:CalculatingTraces})
for the specialization $\tau = (q,t)$. 

\begin{itemize}
\item It follows from the descriptions of the coordinate rings of the affine open sets $U_{\mu}$, 
and the $T$-actions on their coordinates as in (\ref{E:standardopen}), 
that $\Tr_{k(I_{\mu})}(q,t) = n$. 

\item By~\cite[Proposition 5.4.1]{HaimanSurvey}, we know that 
$\Tr_{\mathcal{P}_{I_\mu}}(q,t)$ is the modified Macdonald polynomial.

\item 
It follows immediately from the construction of $\mathcal{V}$ that 
\[ 
\mathsf{tr}_{\mathcal{V}_{I_\mu}}(q,t) = [t]_{n}[q]_{n}.
\]  

\item 
The final term involving the alternating sum of traces equals 
$(1-t)(1-q) \prod_{\mu}$ (see~\cite[Theorem 2]{HaimanDiscrete}).
\end{itemize}

Now, by putting these identities together on the right hand side of (\ref{E:CalculatingTraces}), 
we see that the equivariant Euler characteristic of $\mcal{F}$ equals $n$ times the
Frobenius characteristic of $(-1)^{n-1} \nabla p_{n}$.

\begin{Remark}
If we use the sheaf $\mathcal{F}^{\prime}$ instead of $\mcal{F}$, then 
on the right-hand side of equation (\ref{E:CalculatingTraces}) we get  
$\mathsf{tr}_{\mathcal{V}_{I_\mu}^{\prime}}(q,t)$ instead of $\Tr_{\mcal{V}}(q,t)$. 
This contributes the factor $[t]_{n}[q]_{n}/n$.
\end{Remark}

\subsection{Step 3.}

In the previous subsection, we calculated the equivariant Euler characteristic 
\[\sum_{i} (-1)^{i}
\mathsf{tr}_{H^{i}(H_{n},\mcal{F} \otimes \mcal{P})}(\tau)
\] 
at $\tau = (q,t)$. 
In this section, we will prove a cohomology vanishing result, which will show that the summands 
$\Tr_{H^i(H_n, \mcal{F}\otimes \mcal{P})}$ vanish for $i > 0$. In particular, this will imply that, 
under the derived equivalence, $\Phi(\mcal{F} \otimes \mcal{P})$ is an honest $S_n$-representation 
(and not just a virtual representation) with the expected Frobenius character series.

\begin{Proposition}\label{P:vanishing}
For all $i > 0$, the higher cohomology groups vanish,
\begin{align}\label{A:vanish}
H^{i}(H_{n}, \mcal{F} \otimes \mcal{P}) = 0.
\end{align}
Furthermore, the vanishing result still holds true if $\mcal{F}$ is replaced by $\mcal{F}^{\prime}$ in (\ref{A:vanish}).
\end{Proposition}


\begin{proof}
  Let $\Gamma$ denote the global section functor on $H_n$. We set $\mcal{G} \defeq \mcal{F}\otimes \mcal{P}$ and
  $\mcal{G}^\bullet \defeq \mcal{S}^\bullet \otimes^{L} \mcal{P}$. Passing to the derived category we have isomorphisms
  \begin{align}\label{E: cohomologie}
    \RR\Gamma(\mcal{G}) = \RR\Gamma(\mcal{G}^\bullet).
  \end{align}
  The terms of the complex $\mcal{G}^\bullet$ are direct summands of 
  $(\mcal{P} \otimes \mcal{B}^{\otimes l})^{n^2}$. 
  Since cohomology commutes with direct sums, it follows from~\cite[Theorem 2.1]{HaimanInven} that
  $\mcal{G}^\bullet$ is a complex of acyclic objects for $\Gamma$. 
  Hence, we have
  \[
  \RR\Gamma(\mcal{G}) = \Gamma(\mcal{G^\bullet}). 
  \]
  The complex $\Gamma(\mcal{G}^\bullet)$ has zero cohomology in positive degrees. 
  So, we get $H^i(H_n, \mcal{G}) = 0$ for $i > 0$. 
  Since $H^i(H_n, \mcal{G}) = 0 $ for $i < 0$, the proof of the first claim is complete. 
  
 For the proof of the second claim, we apply the same arguments.
  
\end{proof}

\bibliography{ReferenceNabla}
\bibliographystyle{alpha}

\end{document}